\documentclass[12pt]{amsart}
\usepackage{amsmath, amssymb, amsthm, graphics, colordvi}
\usepackage{geometry}
\usepackage{marginnote}

\usepackage{tikz}
\usepackage{tikz-cd}

\usepackage[all]{xy}
\usepackage{hyperref}
\usepackage{cleveref}

\newtheorem{defn0}{Definition}[section]
\newtheorem{prop0}[defn0]{Proposition}
\newtheorem{thm0}[defn0]{Theorem}
\newtheorem{lemma0}[defn0]{Lemma}
\newtheorem{corollary0}[defn0]{Corollary}
\newtheorem{example0}[defn0]{Example}
\newtheorem{remark0}[defn0]{Remark}
\newtheorem{conjecture0}[defn0]{Conjecture}

\newenvironment{definition}{\medskip \begin{defn0}}{\end{defn0}}
\newenvironment{proposition}{\medskip \begin{prop0}}{\end{prop0}}
\newenvironment{theorem}{\medskip \begin{thm0}}{\end{thm0}}
\newenvironment{lemma}{\medskip \begin{lemma0}}{\end{lemma0}}
\newenvironment{corollary}{\medskip \begin{corollary0}}{\end{corollary0}}
\newenvironment{example}{\medskip \begin{example0}\rm}{\end{example0}}
\newenvironment{remark}{ \medskip\begin{remark0}\rm}{\end{remark0}}
\newenvironment{conjecture}{\medskip\begin{conjecture0}}{\end{conjecture0}}

\newcount\minleft
\newcount\timehour
\def\thetime{\timehour=\time
\divide\timehour by60 \minleft=\timehour \multiply\minleft by -60
\advance\minleft by\time \ifnum\time>720\advance\timehour
by-12\fi\relax
\number\timehour:\ifnum\minleft<10 
    0\fi\relax\number\minleft
    \ifnum\time>720~pm \else~am\fi}

\newcommand{\m}{\mathfrak m}
\def\n{{\mathfrak n}}                   
\def\res{{\bf k}}

\def\gcl{\operatorname{gcl}}
\def\gr{\operatorname{gr}}

\def\HF{\operatorname{HF}}

\def\Syz{\operatorname{Syz}}
\def\Hilb{\operatorname{Hilb}}
\newcommand{\h}{\operatorname{ht}}
\newcommand{\sgn}{\operatorname{sgn}}

\def\ord{{\mathrm{ ord}}}
\def\ch{{\mathrm{ char}}}
\def\supp{{\mathrm{Supp}}}
\def\otau{{\overline{\tau}}}

\newcommand{\Lt}{\operatorname{Lt_{\overline{\tau}}}}

\newcommand{\LC}{\operatorname{LC_{\overline{\tau}}}}

\newcommand{\LF}{\operatorname{LF_{(\overline{\tau},\mathcal{F})}}}

\newcommand{\Ltlex}{\operatorname{Lt_{lex}}}
\newcommand{\Ltau}{\operatorname{Lt_{\tau}}}
\newcommand{\Ltdeg}{\operatorname{Lt_{deglex}}}
\newcommand{\lex}{\operatorname{lex}}
\newcommand{\Lex}{\operatorname{Lex}}
\newcommand{\Gin}{\operatorname{Gin}}
\newcommand{\deglex}{\operatorname{deglex}}
\newcommand{\tail}{\operatorname{tail}}

\newcommand{\rk}{\operatorname{rk}}

\newcommand{\I}{\mathcal{I}}

\begin{document}

\title[]{{\bf
Canonical Hilbert-Burch matrices for power series}}

\author[Roser Homs]{Roser Homs}

\address{Roser Homs
\newline \indent Technische Universit\"at M\"unchen, Parkring 13, 85748 Garching bei M\"unchen, Germany
}  \email{{\tt roser.homs@tum.de}}

\author[Anna-Lena Winz]{Anna-Lena Winz}

\address{Anna-Lena Winz
\newline \indent Freie Universit\"at Berlin, Arnimallee 3, 14195 Berlin, Germany
}  \email{{\tt  anna-lena.winz@fu-berlin.de}}

\begin{abstract} 

Sets of zero-dimensional ideals in the polynomial ring $\res[x,y]$ that share the same leading term ideal with respect to a given term ordering are known to be affine spaces called Gr\"obner cells. Conca-Valla and Constantinescu parametrize such Gr\"obner cells in terms of certain canonical Hilbert-Burch matrices for the lexicographical and degree-lexicographical term orderings, respectively. 

In this paper, we give a parametrization of $(x,y)$-primary ideals in Gr\"obner cells which is compatible with the local structure of such ideals. More precisely, we extend previous results to the local setting by defining a notion of canonical Hilbert-Burch matrices of zero-dimensional ideals in the power series ring $\res[\![x,y]\!]$ with a given leading term ideal with respect to a local term ordering.
\end{abstract}


\maketitle

\section{Introduction}

Punctual Hilbert schemes $\Hilb^d\left(\res[\![x_1,\dots,x_n]\!]\right)$ parametrize points of multiplicity $d$ at the origin. Its counterpart in commutative algebra are local Artinian $\res$-algebras of length $d$. 
These objects have been widely studied in the literature, see \cite{Bri77},\cite{Iar77},\cite{Poo08a}.
By Cohen's structure theorem, any such algebra is of the form $R/J$, where $R=\res[\![x_1,\dots,x_n]\!]$ denotes the power series ring and $J$ is a primary ideal over $\m=(x_1,\dots,x_n)$. 
Moreover, we have the isomorphism $R/J\simeq P/I$, where $I=J\cap P$ is an $\m$-primary ideal in the polynomial ring $P=\res[x_1,\dots,x_n]$. We denote by $\m$ both the unique maximal ideal of $R$ and the homogeneous maximal ideal of $P$, since it will be clear from the context.

We focus on the codimension 2 case, namely $n=2$. Let $E$ be an $\m$-primary monomial ideal. Fix a term ordering $\tau$ on the polynomial ring $P=\res[x,y]$ and denote by $\Ltau(I)$ the leading term ideal of $I$ (see \Cref{def:lt}). 
The sets of zero-dimensional ideals $V_\tau(E)=\lbrace I\subset P: \Ltau(I)=E\rbrace$ are known to be affine spaces. This has been proved for particular term orderings in \cite{Bri77}, \cite{Iar77} and the general case follows from a result by Bia\l{}ynicki-Birula in \cite{Bia73}.
By analogy to Schubert cells in Grassmanians, the affine varieties $V_\tau(E)$ are called Gr\"obner cells.

In \cite{CV08}, Conca and Valla provide a parametrization of Gr\"obner cells in terms of canonical Hilbert-Burch matrices (see \Cref{canonical}) for the lexicographical term ordering. In \cite{Con11}, Constantinescu analogously parametrizes $V_\tau(E)$ in the case of the degree lexicographical term ordering when $E$ is a lex-segment ideal (see (\ref{eq:lex})).

When we restrict our attention to $\m$-primary ideals of $V_\tau(E)$, a major drawback of this construction is that taking leading term ideals is not compatible with the local structure of $A=P/I$ in the sense that Hilbert functions are not preserved. More precisely, for any local ring $(A,\n)$ the Hilbert function of $A$ is defined as the Hilbert function of its associated graded ring $\gr_{\n}(A)=\bigoplus_{t\geq 0}\n^t/\n^{t+1}$. We recall that $\gr_{\n}(A)\simeq P/I^\ast$ (see \cite[Proposition 5.5.12]{GP08}), where $I^\ast$ denotes the initial ideal of $I$ (see \Cref{def:initialIdeal}). Since in general $\Ltau(I)\neq\Ltau(I^\ast)$, the Hilbert function of $P/I$ does not necessarily match the Hilbert function of $P/\Ltau (I)$ (see \Cref{ex:1}). In other words, ideals with different Hilbert functions may belong to the same Gr\"obner cell (see \Cref{ex:CV-vs-local}). 

We can overcome this problem by working in the power series ring. The notion of monomial term ordering is also applicable to a power series ring via local term orderings $\otau$ induced by usual term orderings $\tau$ in the polynomial ring, see \Cref{SS:local_ordering}. We now have the equality $\Lt(J)=\Ltau(J^\ast)$ (see \Cref{rk:HF}), hence Hilbert functions are preserved by taking leading term ideals with respect to local term orderings.

This paper is devoted to the extension of the results by Conca-Valla and Constantinescu to the local setting. Although local term orderings are no longer well-orderings, there are analogous tools to Gr\"obner bases and Buchberger algorithm in the ring of formal power series and in localizations of polynomial rings: standard bases and the tangent cone algorithm, see \cite{Hir64},\cite{Mor82},\cite{GP08} and \Cref{S:local}. 
The resulting Gr\"obner cells $V(E)=\lbrace J\subset R: \Lt(J)=E\rbrace$ are indeed compatible with the local structure.

Combining \Cref{Param} and \Cref{polysurj}, we provide a surjection from a certain set of Hilbert-Burch matrices to the affine variety $V(E)$.
Our main result, \Cref{ParamLex}, gives a parametrization of the ideals in $V(E)$ in terms of their canonical Hilbert-Burch matrices for a special class of monomial ideals $E$. This class includes all lex-segment ideals. By Macaulay's theorem (see \cite{Mac27}; for a modern treatment, see \cite[Theorem 2.9]{Abe12}), any admissible Hilbert function $h$ can be realized by a lex-segment ideal $\Lex(h)$.
In fact, in characteristic zero, $V(\Lex(h))$ parametrizes all ideals with Hilbert function $h$ up to a generic change of coordinates, see \Cref{Cor:hilbertfunction}.
Plenty of examples are given to illustrate the behavior.

In this way, we also generalize the procedure given by Rossi and Sharifan in \cite[Remark 4.7, Example 4.8]{RS10} to explicitly realize ideals with any admissible sequence of zero and negative cancellations on the minimal free graded resolution of $\res[x,y]/\Lex(h)$. In \cite{RS10} the authors considered very specific deformations of a Hilbert-Burch matrix of $\Lex(h)$, in the present paper we parametrize all possible deformations.

Finally, in \Cref{Conj} we point out what should be the set of matrices giving a parametrization of the Gröbner cell $V(E)$ when we drop the lex-segment assumption on $E$. An interesting application of a full parametrization is the computation of all Gorenstein rings that are at a minimal distance of a given Artin ring, see \Cref{S:Gorenstein} and \cite{EHM20}.

\section{Parametrization of ideals in $\res[x,y]$}\label{S:polys}

In the present section we review the parametrization of Gr\"obner cells in $P=\res[x,y]$ in terms of Hilbert-Burch matrices given by Conca-Valla in \cite{CV08} and Constantinescu in \cite{Con11}. 
Let $\res$ be an arbitrary field. Given a monomial term ordering \(\tau\) on a polynomial ring \(P\) over \(\res\) and an ideal \(I \subset P\), the leading term ideal of \(I\) is defined as follows. 

\begin{definition}\label{def:lt}
The \textbf{leading term ideal} $\Ltau(I)$ of the ideal $I\subset P$ with respect to a monomial term ordering $\tau$ in the polynomial ring is the monomial ideal generated by all leading terms of elements in $I$, i.e. $\Ltau(I)=\langle \Ltau(f):f\in I\rangle$.
\end{definition}

In the present paper, we will consider the lexicographical term ordering ($\lex$) and the degree-lexicographical term ordering ($\deglex$) in $P=\res[x,y]$. Recall that with the former we first compare the exponents of $x$ of two monomials, whereas with the latter we first compare their degree. Note that in a polynomial ring in two variables the lexicographical term ordering is equivalent to the reverse lexicographical term ordering.

Consider a monomial zero-dimensional ideal $E$ in $P$. By taking the smallest integer $t$ such that $x^t\in E$ and the smallest integers $m_i$ such that $x^{t-i}y^{m_i}\in E$ for any $1\leq i\leq t$, we can always express such a monomial ideal as

\begin{equation}\label{eq:lex}
  E=(x^t,x^{t-1}y^{m_1},\dots,x^{t-i}y^{m_i},\dots,y^{m_t}),  
\end{equation}

\noindent where $0=m_0< m_1\leq\dots\leq m_t$ is an increasing sequence. 
If all the inequalities are strict, we call \(E\) a \textbf{lex-segment ideal}.

After fixing a term order, we can ask for all ideals $I$ in $P$ with leading term ideal $E$. Reduced Gr\"obner bases provide a parametrization of this set of ideals. However, explicitly describing such a parametrization is not always straightforward. In \cite{CV08}, Conca and Valla consider a different approach: instead of focusing on the generators of $I$, they study the relations or syzygies among the generators. A Hilbert-Burch matrix of the ideal $I$ encodes these relations. Therefore, giving such a parametrization is equivalent to choosing a canonical Hilbert-Burch matrix for each ideal $I$. 

\begin{definition}\label{canonical}
The \textbf{canonical Hilbert-Burch matrix}\index{canonical Hilbert-Burch matrix} of the monomial ideal $E=(x^t,\dots,x^{t-i}y^{m_i},\dots,y^{m_t})$ is the Hilbert-Burch matrix of $E$ of the form

$$H=\left(\begin{array}{cccc}
y^{d_1} & 0 & \cdots & 0\\
-x & y^{d_2} & \cdots & 0\\
0 & -x & \cdots & 0\\
\vdots & \vdots & & \vdots\\
 0 & 0 & \cdots & y^{d_t}\\
0 & 0 & \cdots & -x
\end{array}\right),$$
\noindent
where $d_i=m_i-m_{i-1}$ for any $1\leq i\leq t$. 
The \textbf{degree matrix}\index{degree matrix} $U$ of $E$ is the $(t+1)\times t$ matrix with integer entries $u_{i,j}=m_j-m_{i-1}+i-j$, for $1\leq i\leq t+1$ and $1\leq j\leq t$.
\end{definition}

It follows from the definition that $u_{i,i}=d_i$ and $u_{i+1,i}=1$, for $1\leq i\leq t$.

\medskip

Conca-Valla parametrize the set \(V_0(E)\) of all zero-dimensional ideals $I$ in $P$ that share the same leading term ideal $E$ with respect to the lexicographical term ordering. 
They give a set of matrices that deform the canonical Hilbert-Burch matrix of the monomial ideal $E$ into Hilbert-Burch matrices of each $I$. We use the same notation as in \cite{CV08}.

\begin{definition}\label{def:setT2}
We denote by $T_0(E)$ the set of matrices $N=(n_{i,j})$ of size $(t+1)\times t$ with entries in $\res[y]$ such that 
\begin{itemize}
\item $n_{i,j}=0$ for any $i<j$, 
\item $\deg(n_{i,j})<d_j$ for any $i\geq j$.
\end{itemize}
\end{definition}

\begin{theorem}\label{ThCV08}\cite[Theorem 3.3, Corollary 3.1]{CV08} Given a zero-dimensional monomial ideal $E$ in $P=\res[x,y]$ with canonical Hilbert-Burch matrix $H$, the map
$$\begin{array}{rrcl}
\Phi: & T_0(E) & \longrightarrow & V_0(E)\\
& N & \longmapsto & I_t(H+N)
\end{array}$$

\noindent 
is a bijection. 
\end{theorem}

This theorem allows us to define the canonical Hilbert-Burch matrix of any zero-dimensional ideal $I$ of $P$ as $H+\Phi^{-1}(I)$, where $H$ is the canonical Hilbert-Burch matrix of the monomial ideal $\Ltlex(I)$.

\medskip

In \cite{Con11}, Constantinescu parametrizes the variety 
$$V_{\deglex}(E)=\lbrace I\subset P:\Ltdeg(I)=E\rbrace,$$ 
where the leading term ideals are considered with respect to the degree-lexicographical term ordering, for $E$ lex-segment ideal. 

\begin{definition}
Denote by $\mathcal{A}(E)$ the set of $(t+1)\times t$ matrices $A=(a_{i,j})$
with entries in $\res[y]$ such that all its entries satisfy
\[\deg(a_{i,j})\leq \begin{cases}
\min(u_{i,j}+1, d_i-1), & i\leq j;\\
\min(u_{i,j},d_j-1), & i > j;
\end{cases}\]
\noindent
and $u_{i,j}$ are the entries of the degree matrix $U$ of $E$.
\end{definition}

\begin{theorem}\label{ThCon11}\cite[Theorem 3.1]{Con11} Given a zero-dimensional lex-segment ideal $L$ in $P=\res[x,y]$ with canonical Hilbert-Burch matrix $H$, the map
$$\begin{array}{rrcl}
\Phi: & \mathcal{A}(L) & \longrightarrow & V_{\deglex}(L)\\
& A & \longmapsto & I_t(A+H)
\end{array}$$

\noindent 
is a bijection.
\end{theorem}

The proofs of well-definedness and surjectivity of $\Phi$ in \cite{Con11} hold for any monomial ideal and, although the lex-segment hypothesis is needed in his proof of injectivity, the author conjectures that $\Phi$ is a proper parametrization in the general case.

\section{From polynomials to power series}\label{S:local}

We are interested in a construction of Gr\"obner cells for $\m$-primary ideals of $P=\res[x,y]$ in the same spirit as the ones presented in \Cref{S:polys} but compatible with the local structure in the sense described in the Introduction. So from now on we will work in the ring of formal power series $R=\res[\![x,y]\!]$. 

Zero-dimensional monomial ideals of $R$ can still be described as $E=(x^t,\dots,x^{t-i}y^{m_i},\dots,y^{m_t})$ and we can define their canonical Hilbert-Burch matrix $H$ as introduced in \Cref{canonical}. 

The goal of this section is to provide the necessary tools to extend the strategies in the proofs by Conca-Valla and Constantinescu to the local setting.
In the first part, we define a local term ordering $\otau$ and the notion of $\otau$-enhanced standard basis, the local analogous to Gr\"obner basis. The second part is devoted to the lifting of syzygies.

\subsection{Enhanced standard basis and Grauert's division}\label{SS:local_ordering}

\begin{definition}
A term ordering $\tau$ in the polynomial ring $P=\res[x_1,\dots,x_n]$ induces a reverse-degree ordering $\otau$ in $R=\res[\![x_1,\dots,x_n]\!]$ such that for any monomials $m,m'$ in $R$, $m>_{\otau} m'$ if and only if
$$\deg(m) < \deg(m')$$
or
$$\deg(m) = \deg(m') \mbox{ and } m >_{\tau} m_0.$$
\noindent
We call $\otau$ the \textbf{local term ordering} induced by the global term ordering $\tau$. 
\end{definition} 

Note that the local term orderings induced by the lexicographical and the degree lexicographical term orderings are the same.

\begin{definition} Given an ideal $J$ of $R$, we define the \textbf{leading term ideal} of $J$ as the monomial ideal in $P$ generated by the leading terms with respect to the local term ordering $\otau$, i.e.
$$\Lt(J)=\left( \Lt(f): f\in J\right)\subset\res[x,y].$$

We call a subset $\lbrace f_1,\dots,f_m\rbrace$ of $J$ a \textbf{$\otau$-enhanced standard basis of $J$} if $\Lt(J)=(\Lt(f_1),\dots,\Lt(f_m))$.
\end{definition}

\begin{definition}\label{def:initialIdeal}
The \textbf{initial form} $f^\ast$ of an element $f$ in $R$ is the homogeneous polynomial consisting of the terms in $f$ of lowest degree. The \textbf{initial ideal} $J^{\ast}$ is the homogeneous ideal generated by the initial forms of elements in $J$. 
\end{definition}
\begin{remark}\label{rk:HF}
Note that, by definition of local term ordering, $\Lt(f)=\Ltau(f^\ast)$. Therefore, $\Lt(J)=\Ltau(J^\ast)$. Let $\HF_{R/J}=h$ denote the Hilbert function of $R/J$. Then
$$\HF_{R/J}=\HF_{P/J^\ast}=\HF_{P/ \Lt(J)}=\HF_{P/\Lex(h)},$$
where $\Lex(h)$ is the unique lex-segment ideal with the same Hilbert function.
\end{remark}

\begin{remark}\label{rk:terminology} The term standard basis was first used by Hironaka in \cite[Definition 3]{Hir64} to refer to systems of generators of the initial ideal $J^\ast$. However, this terminology is not consistent in literature and in other sources standard basis refer to what we here define as $\otau$-enhanced standard basis, e.g. \cite{GP08}. The notation used in this paper is the same as in \cite{Ber09}. 
\end{remark}

\begin{example}\label{ex:1}\emph{Comparison between leading terms w.r.t. global and local term orderings.} Consider the lex-segment ideal $L=(x^3,x^2y,xy^3,y^5)$ and set $\tau=\lex$. Let $H$ be its canonical Hilbert-Burch matrix and $U$ its degree matrix from \Cref{canonical}:

$$H=\left(\begin{array}{ccc}
y & 0 & 0 \\
-x & y^2 & 0\\
0 & -x & y^2 \\
0 & 0 & -x\\
\end{array}\right),\quad U=\left(\begin{array}{ccc}
1 & 2 & 3\\
1 & 2 & 3\\
0 & 1 & 2\\
-1 & 0 & 1\\
\end{array}\right).$$

Consider the matrix $M=H+N$, where $N$ is a $4\times 3$ matrix with all zero entries except for $1$ in the $(4,3)$-entry. From Conca-Valla parametrization in \Cref{ThCV08}, we know that $I=I_3(M)\subset P$ is an ideal in $V_0(L)$. Indeed, the maximal minors of $M$ give a $\tau$-Gr\"obner basis $\lbrace x^3-x^2,x^2y-xy,xy^3-y^3,y^5\rbrace$ of $I$ and $\Ltlex(I)=L$. 

However, the $3\times 3$-minors of $M$ are not a $\otau$-enhanced standard basis of the ideal $J=IR$, namely the extension of $I$ in the power series ring. In fact, $J=(x^2,xy,y^3)$ is itself a lex-segment ideal. The reason why the leading term ideal changes when computed with respect to $\otau$ is that $n_{4,3}=1$ has a term of degree lower than $u_{4,3}=1$. Finally, note that $\Ltau(I)\neq \Ltau(I^{\ast})=\Lt(J)=J$.
\end{example}

Buchberger division can be replaced in the power series ring by Grauert's division, see \cite{Gra72}. Later on, Mora gave an analogous method to Buchberger's algorithm in the local case: the tangent cone algorithm, see \cite{Mor82}. We reproduce next a modern formulation of Grauert's division theorem in $\res[\![x_1,\dots,x_n]\!]$ from \cite[Theorem~6.4.1]{GP08}:

\begin{theorem}\label{Grauert}[Grauert's Division Theorem]
Let $f,f_0,\dots,f_t$ be in $R$. Then there exist $q_0,\dots,q_t,r\in R$ such that $$f=\sum_{i=0}^t q_if_i+r$$ satisfying the following properties:
\begin{enumerate}
\item No monomial of $r$ is divisible by any $\Lt(f_i)$, for $0\leq i\leq t$.
\item If $q_i\neq 0$, $\Lt(q_if_i)\leq_{\otau}\Lt(f)$.
\end{enumerate}
\end{theorem}

These techniques can be used to extend results that are well-understood for graded algebras to the local case. In \cite{ERV14}, they have been successfully applied to characterize the Hilbert function of one dimensional quadratic complete intersections.

\subsection{Lifting of syzygies in local rings}
The connection between the lifting of syzygies and Gr\"obner bases has been widely studied in polynomial rings, see \cite[Theorem 2.4.1]{KR00}. Analogous results hold for rings of formal power series.

Let $\mathcal{F}$ be a subset $\lbrace f_0,\dots,f_t\rbrace$ of $R$ and set $\Lt(\mathcal{F})=\left\{\Lt(f_0),\dots,\Lt(f_t)\right\}$. By a slight abuse of notation, $\mathcal{F}$ and \(\Lt(\mathcal{F})\) will be regarded as $(t+1)$-tuples of $R^{t+1}$ when convenient. 
Mora, Pfister and Traverso prove in \cite[Theorem 3]{MPT89} that $\mathcal{F}$ is a  $\otau$-enhanced standard basis of an ideal of $R$ if and only if any homogeneous syzygy of $\Lt(\mathcal{F})$ can be lifted to a syzygy of $\mathcal{F}$.

For the sake of completeness, we will now give a precise definition of lifting in this setting following the notation of \cite[Definition 1.7]{Ber09}. We define the \textbf{degree} of $m=(m_1,\dots,m_{t+1})\in R^{t+1}$ with respect to the $(t+1)$-tuple $\mathcal{F}\in R^{t+1}$ and the local term ordering $\otau$ as 

$$\deg_{(\otau,\mathcal{F})}(m)=\max_{\otau}\lbrace \Lt(m_i f_{i-1}): 1\leq i\leq t+1\mbox{ and }m_i\neq 0\rbrace.$$

An element $\sigma = \{\sigma_1, \dots, \sigma_{t+1}\} \in R^{t+1}$ is homogeneous with respect to $(\otau,\mathcal{F})$-degree if all its non-zero components reach the maximum leading term, namely $\Lt(\sigma_i f_{i-1})=\deg_{(\otau,\mathcal{F})}(\sigma)$ for any $i\in\lbrace 1,\dots,t+1\rbrace$ such that $\sigma_i\neq 0$.

\begin{definition} We call $m\in R^{t+1}$ a \textbf{$(\otau,\mathcal{F})$-lifting} of a $(\otau,\mathcal{F})$-homogeneous element $\sigma\in R^{t+1}$ if $m=\sigma + n$, where $n=(n_1,\dots,n_{t+1})\in R^{t+1}$ satisfies 
\begin{equation}\label{eq:lifting}
\Lt(n_if_{i-1})<_{\otau} \deg_{(\otau,\mathcal{F})}(\sigma)
\end{equation}
for any $1\leq i\leq t+1$ such that $n_i\neq 0$.
Conversely, we call $\sigma$ the \textbf{$(\otau,\mathcal{F})$-leading form} of $m$ and denote it by
$\LF(m)=\sigma\in R^{t+1}$.
\end{definition}

If both $\otau$ and $\mathcal{F}$ are clear from the context, we will just say that $m$ is a lifting of $\sigma$, which in its turn is the leading form of $m$. The shift on the indexes of $n$ and $\mathcal{F}$ in (\ref{eq:lifting}) is convenient for our specific setting, as we will see in the following example.

\begin{example}{\emph{Liftings of homogeneous elements in
\(R\)-free modules.}} \label{Ex:lift}
Consider a monomial ideal $E=(x^t,x^{t-1}y^{m_1},\dots,y^{m_t})$ and take $\mathcal{F}=(f_0,\dots,f_t)\in R^{t+1}$ such that $\Lt(f_i)=x^{t-i}y^{m_i}$ for any $0\leq i\leq t$. The columns $\sigma^1,\dots,\sigma^t$ of the canonical Hilbert-Burch matrix $H$ of $E$ are $(\otau,\mathcal{F})$-homogeneous elements with $\deg_{(\otau,\mathcal{F})}(\sigma^j)=x^{t-j+1}y^{m_j}$ for any $1\leq j\leq t$. We can build liftings $m^j$ of $\sigma^j$ by taking $m^j=\sigma^j+n^j$, where $n^j=(n_{1,j},\dots,n_{t+1,j})$ is a $(t+1)$-tuple of $R^{t+1}$ such that either $n_{i,j}=0$ or $\Lt(n_{i,j})x^{t-i+1}y^{m_{i-1}}<_{\otau} x^{t-j+1}y^{m_j}$.
\end{example}

As in the polynomial case, Bertella proves in \cite[Theorem~1.10]{Ber09} that the module of syzygies of $\mathcal{F}$ is generated by liftings of homogeneous generators of the module of syzygies of $\Lt(\mathcal{F})$. Recall that the fact that syzygies lift is equivalent to the existence of a flat family $I_t$ where $I_0=\Lt(\mathcal{F})$ and $I_1=(\mathcal{F})$, see \cite[Chapter 1]{Ste03} and \cite[Lemma 18.8]{MS05}.

In the same paper, Bertella provides a very explicit characterization of $\otau$-enhanced standard bases in codimension two in terms of matrices that encode leading forms of the generators of the module of syzygies of the ideal:

\begin{theorem}\cite[Theorem~1.11]{Ber09}\label{1.11}
Let $M$ be a $(t+1)\times t$ matrix with entries in $R$. For $0\leq i\leq t$, let $f_i$ be the determinant of $M$ after removing row $i+1$ and set $\mathcal{F}=(f_0,\dots,f_t)$. Let $H$ be the matrix whose columns are the $(\otau,\mathcal{F})$-leading forms of the columns of $M$.
Assume that:
\begin{itemize}
\item $\h(f_0,\dots,f_t)=2$,
    \item $I_t(H)=(\Lt(f_0),\dots,\Lt(f_t))$.
\end{itemize}
Then the following are equivalent:
\begin{enumerate} 
    \item[(i)] $\lbrace f_0,\dots,f_t\rbrace$ is a $\otau$-enhanced standard basis of the ideal $I_t(M)$.
    \item[(ii)] $\h(\Lt(f_0),\dots,\Lt(f_t))=2$.
\end{enumerate}
\end{theorem}

In other words, for zero-dimensional ideals $J$ in $R=\res[\![x,y]\!]$, a $\otau$-enhanced standard basis $\mathcal{F}$ arises from maximal minors of a Hilbert-Burch matrix \(M\) that encodes liftings of syzygies of $\Lt(\mathcal{F})$. 

\section{Towards a parametrization of ideals in $\res[\![x,y]\!]$}
From now on we will consider \(\tau\) to be the lexicographical term ordering.

\begin{definition} \label{def:setV}
Given a zero-dimensional monomial $E$ ideal in $R$, we denote by $V(E)$ the set of ideals $J\subset R$ such that $\Lt(J)=E$. 
\end{definition}

Let us start by defining a set of matrices whose maximal minors generate all the ideals with the same leading term ideal with respect to the local term ordering $\otau$. 

\begin{definition}\label{def:setN}
Let $E$ be a monomial ideal with canonical Hilbert-Burch matrix $H$ and associated degree matrix $U=(u_{i,j})$. We define the set $\mathcal{N}(E)$ of $(t+1)\times t$ matrices $N=(n_{i,j})$ with entries in $\res[\![y]\!]$ such that all its non-zero entries satisfy
\[\ord(n_{i,j})\geq \begin{cases}
u_{i,j}+1, & i\leq j;\\
u_{i,j}, & i > j
\end{cases},\]
where \(\ord(n_{i,j})\) denotes the degree of the initial form of \(n_{i,j}\).
\end{definition}

\begin{theorem}\label{Param} Given a monomial ideal $E=(x^t,\dots,x^{t-i}y^{m_i},\dots,y^{m_t})$ in $R$ with canonical Hilbert-Burch matrix $H$ and degree matrix $U$, let $V(E)$ be the set of ideals in \Cref{def:setV} and let $\mathcal{N}(E)$ be the set of matrices in \Cref{def:setN}. The map
$$\begin{array}{rrcl}
\varphi: & \mathcal{N}(E) & \longrightarrow & V(E)\\
& N & \longmapsto & I_t(H+N)
\end{array}$$
\noindent
is surjective.
\end{theorem}

We prove \Cref{Param} in two steps: well-definedness in \Cref{well} and surjectivity in \Cref{surj}.

\begin{lemma}\label{well} The map $\varphi$ is well-defined.
\end{lemma}

\begin{proof}
We need to prove that the leading term ideal $\Lt\left(I_t(H+N)\right)$ is the mononomial ideal $E$ for any matrix $N=(n_{i,j})$ in the set $\mathcal{N}(E)$.

Consider the matrix $M=H+N$. The order bounds on the entries of $N$ yield 
$$\ord(m_{i,j})\geq \left\lbrace\begin{array}{cc}
u_{i,j}+1, & i<j;\\
u_{i,j}, & i\geq j.
\end{array}\right.$$

Set $f_i=\det[M]_{i+1}$, for any $0\leq i\leq t$, where $[M]_{i+1}$ is the square matrix that we get after removing row $i+1$ of $M$. Since 
$$f_i=\sum_{\sigma\in S_t}\sgn(\sigma)\prod_{1\leq k\leq t+1,\,k\neq i+1}m_{k,\sigma(k)},$$ 
we study the leading terms of polynomials of the form $h=\prod_{1\leq k\leq t+1,\,k\neq i+1}m_{k,\sigma(k)}$.

If $h$ is the product of all elements in the main diagonal of $[M]_{i+1}$, then $\Lt(h)=x^{t-i}y^{m_i}$. We claim that any other $h\neq 0$ satisfies $\Lt(h)<_{\otau} x^{t-i}y^{m_i}$. Indeed, since
$$\Lt(h)=\prod_{1\leq k\leq t+1,\,k\neq i+1}\Lt(m_{k,\sigma(k)}),$$
\noindent
then
$$\ord(h)=\sum_{1\leq k\leq t+1,\,k\neq i+1}\ord(m_{k,\sigma(k)})\geq \sum_{1\leq k\leq t+1,\,k\neq i+1} u_{k,\sigma(k)}.$$

Equality can only be reached if subindices $(i,j)$ satisfy $i\geq j$, namely
$$h=\prod_{k=1}^i(y^{d_k}+n_{k,k})\prod_{k=i+1}^{t+1}m_{k,\sigma(k)},$$
\noindent
hence the maximal power of $x$ is only reached at the main diagonal. Thus, any $h\neq 0$ away from the main diagonal satisfies $\Lt(h)<_{\otau}x^{t-i}y^{m_i}$ and, therefore, $\Lt(f_i)=x^{t-i}y^{m_i}$.

Now we need to show that $\{f_0,\dots,f_t\}$ forms a $\overline{\tau}$-enhanced standard basis of $I_t(M)$. From the order bounds on the entries $n_{i,j}$ of $N$, it follows that the columns of $M$ are liftings of the columns of $H$. See \Cref{Ex:lift} for more details.
By \Cref{1.11}, it is enough to show that $\h\left((\Lt(f_0),\dots,\Lt(f_t))\right)=2$,
which is clear because this ideal contains pure powers $x^t$ and $y^{m_t}$. Therefore, $\Lt\left(I_t(M)\right)=E$. 
\end{proof}

\begin{lemma}\label{surj}
The map $\varphi$ is surjective.
\end{lemma}

\begin{proof}

Consider a $\otau$-enhanced standard basis $\lbrace f_0,\dots,f_t\rbrace$ of $J\in V(E)$ such that $\Lt(f_i)=x^{t-i}y^{m_i}$. We can assume that the monomials in the support of the $f_i$'s are not divisible by $x^t$, except for $\Lt(f_0)$. 

For any $1\leq j\leq t$, consider the $S$-polynomials $S_j:=S(f_{j-1},f_j)=y^{d_j}f_{j-1}-xf_j$. Note that no monomial in $\supp(S_j)$ is divisible by $x^{t+1}$ for any $1\leq j\leq t$. By \Cref{Grauert} we have
$$S_j=\sum_{i=0}^{t}q_{i,j}f_i,$$
for some $q_{i,j}\in\res[\![x,y]\!]$ such that $\Lt(q_{i,j}f_i)\leq\Lt(S_j)$. 
We claim that $q_{i,j}\in\res[\![y]\!]$.

In fact, we will prove that this holds for any $f\in J$ such that $x^{t+1}$ does not divide any monomial in $\supp(f)$. Assume $\LC(f)=1$. Consider such an $f$, then $\Lt(f)=x^sy^r$ for some $0\leq s\leq t$. On the other hand, from the fact that $\Lt(f)$ belongs to $\Lt(J)$, it follows that $x^{t-i}y^{m_i}$ must divide $\Lt(f)$ for some $0\leq i\leq t$. Then $t-i\leq s$ and $m_i\leq r$, hence $m_{t-s}\leq m_i\leq r$. Define

$$g=f-y^{r-m_{t-s}}f_{t-s}.$$

The new element $g$ still belongs to $J$ and satisfies again that none of its monomials is divisible by $x^{t+1}$. In this way we can define a sequence $({g_i})_{i\in\mathbb{N}}$, starting by $g_0=f$, whose elements have decreasing leading terms with respect to $\otau$. 
As in the proof of Grauert's division theorem in \cite[Theorem~6.4.1]{GP08}, $\sum_{i\in\mathbb{N}} g_k$ converges with respect to the $\m$-adic topology and
$$f=\sum_{k\in\mathbb{N}}(g_k-g_{k+1})=\sum_{i=0}^t \left(\sum_{k\in\mathbb{N}, s_k=t-i}y^{r_k-m_{t-s_k}}\right)f_i.$$

Therefore, for any $1\leq j\leq t$, the $S$-polynomial $S_j$ provides a relation between generators of $J$
$$y^{d_j}f_{j-1}-xf_j+\sum_{i=1}^{t+1}n_{i,j}f_{i-1}=0,$$
where $n_{i,j}=-q_{i-1,j}\in\res[\![y]\!]$. This expression can be encoded in the matrix $M=H+N$, where $N=(n_{i,j})$. 
From $\Lt(n_{i,j}f_{i-1})\leq_{\overline{\tau}}\Lt(S_j)$ it follows that any column $m^i$ of $M$ is a lifting of a column $\sigma^i$ of $H$. 
The columns $\sigma^1,\dots,\sigma^t$ of $H$ constitute a homogeneous system of generators of $\Syz(\Lt(J))$. 
Then, by \cite[Theorem~1.10]{Ber09}, $m^1,\dots,m^t$ generate $\Syz(J)$. The Hilbert-Burch theorem ensures that $J$ is generated by the maximal minors of $M$.

Finally, the order bounds on the entries of $N$ are obtained again from $\Lt(n_{i,j}f_{i-1})\leq_{\overline{\tau}}\Lt(S_j)$.
Indeed, $x^{t-i+1}y^{m_{i-1}+\beta_{i,j}}<_{\otau}x^{t-j+1}y^{m_j}$,
where $\Lt(n_{i,j})=y^{\beta_{i,j}}$. Since 
\begin{equation}\label{eq:surj}
\beta_{i,j}+t-i+1+m_{i-1}\geq t-j+1+m_j, 
\end{equation}
we have $\beta_{i,j}\geq i-j+m_j-m_{i-1}=u_{i,j}$. If $\beta_{i,j}=u_{i,j}$, then equality holds in (\ref{eq:surj}) and hence $t-i+1<t-j+1$. In other words, $\beta_{i,j}\geq u_{i,j}$ and equality is only reachable when $i>j$.
\end{proof}

The proof of \Cref{surj} provides a constructive method to obtain a matrix $N\in\mathcal{N}(E)$ from any $\otau$-enhanced standard basis $\lbrace f_0,f_1,\dots,f_t\rbrace$ of $J\in V(E)$ such that $\Lt(f_i)=x^{t-i}y^{m_i}$ and $x^t$ does not divide any term of any $f_i$ except for $\Lt(f_0)$.  

\begin{example} \label{ex:polysurj}
{\emph{Matrices in $\mathcal{N}(E)$ with power series entries.}} Set $J=(x^4+x^3y,y^2+x^3+x^2y)$ and consider the $\otau$-enhanced standard basis

$$\begin{array}{l}
f_0=x^4+x^3y,\\
f_1=x^3y^2+y^5,\\
f_2=x^2y^2,\\
f_3=xy^2,\\
f_4=y^2+x^3+x^2y.
\end{array}$$
\noindent

It can be checked that it satisfies the conditions of \Cref{surj}. The first $S$-polynomial is $y^2f_0-xf_1=\left(\sum_{i\geq 1}y^i\right)f_1+\left(\sum_{i\geq 3}y^i\right)f_2-y^3f_3-\left(\sum_{i\geq 4}y^i\right)f_4$, hence some entries in $N$ are proper power series, not polynomials.
\end{example}

Next we will see that, for any ideal $J\in V(E)$, we can always find a matrix $N\in\mathcal{N}(E)$ with polynomial entries such that $\varphi(N)=J$.  

\begin{example} \label{ex:polysurj2} {\emph{Matrices in $\mathcal{N}(E)$ with polynomial entries.}}
The matrix in $\mathcal{N}(E)$ obtained from the $\otau$-enhanced standard basis of $J=(x^4+x^3y,y^2+x^3+x^2y)$ given in \Cref{ex:polysurj} is
\[
N=\left(\begin{array}{cccc}
0 & 0 & 0 & 1\\
- \sum_{i\geq 1} y^i & \sum_{i\geq 1} y^i & 0 & 0\\
- \sum_{i\geq 3} y^i & \sum_{i\geq 2} y^i & 0 & 0\\
y^3 & 0 & 0 & 0\\
\sum_{i\geq 4} y^i & - \sum_{i\geq 3} y^i & 0 & 0
\end{array}\right).
\]
By removing all the terms of degree larger than 3 we get the matrix
\[
\overline{N}=\left(\begin{array}{cccc}
0 & 0 & 0 & 1\\
- y-y^2-y^3 & y+y^2+y^3 & 0 & 0\\
- y^3 & y^2+y^3 & 0 & 0\\
y^3 & 0 & 0 & 0\\
0 & - y^3 & 0 & 0
\end{array}\right)
\]
with polynomial entries. Check that \(J =\varphi(N)=\varphi(\overline{N})\). Observe that, although the behaviour with respect to the syzygies is much better, the $\otau$-enhanced standard basis of $J$ given by the minors of \(H+\overline{N}\) is less simple, for example 
\(\bar{f}_0 = x^4+x^3y+y^4-xy^4+y^5-x^2y^4-xy^5+y^6-x^2y^5-xy^6\).
\end{example}

For a general $J\in V(E)$, we can only ensure that we will obtain the same ideal if we remove the terms in the entries of $N$ with degree strictly higher than the socle degree of $R/J$, namely the largest integer $s$ such that $\m^{s+1}\subset J$.

\begin{definition}\label{def:setNs} Let $E$ be a monomial ideal and let $s$ be the socle degree of $R/E$. We define the set of matrices \(\mathcal{N}(E)_{\le s} :=
\mathcal{N}(E) \cap (\res[\![y]\!]_{\le s})^{(t+1) \times t}\).
\end{definition}

\begin{proposition}\label{polysurj} 
The restriction of \(\varphi\) to \(\mathcal{N}(E)_{\le s}\)
is surjective.
\end{proposition}
\begin{proof}
Consider \(J \in V(E)\), by \Cref{surj} we know that \(J=I_t(H+N)\) for some $N\in \mathcal{N}(E)$.
Recall that $J$ has the same Hilbert function as $E$, hence the socle degree of $J$ is also $s$. We express $N$ as \(N = \overline{N} +  \widetilde{N}\), where \(\overline{N} \in \mathcal{N}(E)_{\le s}\) and \(\widetilde{N} \in (\res[\![y]\!]_{\ge s+1})^{(t+1) \times t}\).
We decompose \(\widetilde{N}\) into matrices \(\widetilde{N}_{i,j}\) with at most one non-zero entry at position \((i,j)\) such that 
\( \widetilde{N} = \sum_{i = 1,\dots t+1,j= 1,\dots,t} \widetilde{N}_{i,j}\).

By definition, $J=(f_0,\dots,f_t)$, where \(f_k =  \det([H +N]_{k+1})\). Our goal is to prove that $J=(\bar{f}_0,\dots,\bar{f}_t)$, where \(\bar{f}_k =  \det([H+ \overline{N}]_{k+1})\).

Let us use the Laplacian rule to rewrite the determinant. 
We denote by \(\left[M\right]_{(l,m),n}\) the (square) submatrix of \(M\) that is obtained by deleting the \(l\)-th and \(m\)-th rows and the \(n\)-th column. Then
\[
\begin{array}{ll}
f_k &= \det\left(\left[H+\overline{N}+ \sum_{i,j} \widetilde{N}_{i,j}\right]_{k+1}\right)\\
&= \det\left(\left[H+\overline{N}\right]_{k+1}\right) + 
\sum_{i,j} \pm \tilde{n}_{i,j} \cdot \det\left(\left[H+\overline{N}\right]_{(k+1,i),j}\right)\\
&= \bar{f}_k + \sum_{i,j} \pm \tilde{n}_{i,j} \cdot \det\left(\left[H+\overline{N}\right]_{(k+1,i),j}\right).
\end{array}
\]
Since \(\tilde{n}_{i,j} \in \res[\![y]\!]_{\ge s+1}\), it is clear that $f_k-\bar{f}_k\in (x,y)^{s+1} \subset J$. Then \(J'=\left(\bar{f}_0, \dots, \bar{f}_t\right)\subset J\) and, because $\Lt(J')=\Lt(J)$, we deduce that \(J=\left(\bar{f}_0, \dots, \bar{f}_t\right)\).
\end{proof}

It is important to note that \Cref{polysurj} does not provide a parametrization of $V(E)$. In general, the map $\varphi$ is not injective even when we restrict it to $\mathcal{N}(E)_{\le s}$.

\begin{example}\label{ex:noninj} {\emph{The restriction of $\varphi$ is not injective.}} Continuing \Cref{ex:polysurj} and \Cref{ex:polysurj2}, note that $\overline{N}\in\mathcal{N}_{\le 4}(E)$ but also  
\[
N'=\left(
\begin{array}{llll}
0& 0&     0&1\\
-y&0&     0&0\\
0&0&    0&0\\
0&0&     0&0\\
0&0&0&0 
\end{array}
\right)\in\mathcal{N}_{\le 4}(E),
\]
with \(\varphi(N')=\varphi(\overline{N})=J\).

The corresponding associated $\otau$-enhanced standard basis of $J$ is $\lbrace x^4+x^3y,\  x^3y^2,\ x^2y^2,\ xy^2,\ y^2+x^3+x^2y\rbrace$.
\end{example}

\begin{remark}\label{miniinj}
We have seen that \(\varphi: \mathcal{N}(E) \to V(E)\) as well as its restriction 
\(\varphi: \mathcal{N}_{\le s}(E) \to V(E)\) are not injective. Although an ideal $J$ can be obtained from different matrices of the form $H+N$, the systems of polynomial generators $\lbrace f_0,\dots,f_t\rbrace$ of $J$ that arise as maximal minors of any such matrices are all different. In other words, the map $\mathcal{N}(E) \to R^{t+1}$, that sends $N$ to the maximal minors of $H+N$, is injective.

Indeed, if two matrices \(N, N' \in \mathcal{N}(E)\) satisfy that 
the maximal minors of \(H+N\) and \(H+N'\) coincide, it follows that 
\(N=N'\). 
The argument is the same as in the first paragraph of \cite[3.2]{Con11} and we reproduce it here.
Let \(\{f_0,\dots,f_t\}\) be the maximal minors of \(H+N\) and \(H+N'\). 
The columns of both matrices are syzygies of \(\lbrace f_0,\dots,f_t\rbrace\), thence
the columns of their difference \(H+N-(H+N')=N-N' \in \res[\![y]\!]^{(t+1)\times t}\) are also syzygies, 
but since the leading terms of the \(f_i\) involve different powers of \(x\), it follows that \(N=N'\).
\end{remark}

\section{Parametrization for lex-segment leading term ideals}

A special situation occurs when a $\otau$-enhanced standard basis of $J$ and a Gr\"obner basis of the ideal $I=J\cap P$ with respect to the lexicographical term ordering, shortly denoted \(\lex\)-Gr\"obner basis, coincide. In this setting, we can overcome the lack of injectivity of $\varphi:\mathcal{N}(E)\rightarrow V(E)$ by using Conca-Valla's parametrization of $V_0(E)$.

\begin{proposition}\label{CVmatrix} Let $J\in V(E)$ be an ideal that admits a $\otau$-enhanced standard basis \(\{f_0,\dots,f_t\}\) 
that is also a \(\lex\)-Gr\"obner basis of $I=J\cap P$
with \(\Lt(f_i)=\Ltlex(f_i)\). Then there exists a unique matrix $N\in\mathcal{N}(E)\cap T_0(E)$ such that $J=I_t(H+N)$.  
\end{proposition}
\begin{proof} 
Let \(\{f_0, \dots, f_t\}\) be a \(\otau\)-enhanced standard basis of \(J\) that is also a \(\lex\)-Gr\"obner basis 
  with \(\Lt(f_i)=\Ltlex(f_i)=x^{t-i}y^{m_i}\).
Then the \(f_i\) are the signed maximal minors of \(H+N\) for some \(N \in \mathcal{N}(E)\) that is a strictly lower triangular matrix with 
polynomial entries. 
Here by strictly lower triangular, we mean that 
\(n_{i,j}=0\) for all \(i \le j\).

Assume that \(N\) is not yet in \(T_0(E)\), 
namely there exist \((i,j)\) with \(\deg(n_{i,j})\ge d_j\). 
In that case we decompose $n_{i,j}=r_{i,j} + y^{d_j} q_{i,j}$ with 
\begin{itemize}\label{OrderBounds}
\item \(u_{i,j}\le \ord(r_{i,j})\le\deg(r_{i,j}) \le d_j-1\),
\item \(\max(u_{i,j}-d_j,0)\le \ord(q_{i,j})\le\deg(q_{i,j}) \le \deg n_{i,j}-d_j\).
\end{itemize}

Next we will perform the 
\((i,j)\)-reduction move defined in \cite[Proof~of~3]{Con11} on \(N\). Note that since \(N\) is strictly lower triangular, it corresponds to the second type of reduction moves:

\begin{itemize}
    \item[Step 1.] Add the \(j\)-th row multiplied by \(-q_{i,j}\) to the 
\(i\)-th row of \(H+N\).
    \item[Step 2.] Add the \((i-1)\)-th column multiplied by \(q_{i,j}\) to the
\((j-1)\)-th column of the matrix resulting from Step 1.
\end{itemize}

This operation does not change the ideal \(J\) and produces a new matrix \(\widetilde{N}\) whose $(i,j)$-entry has degree strictly less than $d_j$. 
Checking that it preserves the order bounds on the entries is a technicality that follows from the order bounds on $r_{i,j}$ and $q_{i,j}$.
Thus the matrix \(\widetilde{N}\) we obtain will still be in \(\mathcal{N}(E)\) and the maximal minors of \(H+\widetilde{N}\)
will form a \(\otau\)-enhanced standard basis of \(J\).

After performing finitely many reduction steps from the last to the first column, we will obtain a matrix
\(N_0 \in T_0(E) \cap \mathcal{N}(E)\)
with \(J = I_t(H+N_0)\). By \Cref{ThCV08}, \(N
_0\) is unique.
\end{proof}

This result allows us to extend the definition of canonical Hilbert-Burch matrix to any ideal that has a \(\otau\)-enhanced standard basis \(\{f_0,\dots,f_t\}\) that satisfies $\Lt(f_i)=\Ltlex(f_i)=x^{t-i}y^{m_i}$. Moreover, the proof of \Cref{CVmatrix} gives an algorithm to construct the canonical matrix from the matrix that encodes the S-polynomials of \(\{f_0,\dots,f_t\}\) via reduction moves.

\begin{definition}\label{def:setM}
Set $\mathcal{M}(E):=\mathcal{N}(E)\cap T_0(E)$. Let $J\in V(E)$ be an ideal that admits a $\otau$-enhanced standard basis which is also a \(\lex\)-Gr\"obner basis of $I=J\cap P$. We define the \textbf{canonical Hilbert-Burch matrix}\index{canonical Hilbert-Burch matrix} of $J$ as \(H+N\), where $N$ is the unique matrix in $\mathcal{M}(E)$ such that $J=I_t(H+N)$.
\end{definition}

\begin{remark}\label{V2}
In \cite{CV08}, Conca and Valla provide parametrizations of certain subsets of $V_0(E)$. $V_2(E)$ is the set of all $(x,y)-$primary ideals $I$ such that $\Ltlex(I)=E$ and it is parametrized by the set of matrices $T_2(E)$ (see \cite[Definition 3.2]{CV08} for an explicit description). It is not difficult to check that $\mathcal{M}(E)=\mathcal{N}(E)\cap T_0(E)=\mathcal{N}(E)\cap T_2(E)$. 
\end{remark}

\medskip

\begin{example}\label{Ex:step1}{\emph{Canonical Hilbert-Burch matrix.}}
Consider $J=(x^6,xy^2-y^5,y^8)$ and $E=\Lt(J)=(x^6$, $x^5y^2$, $x^4y^2$, $x^3y^2$, $x^2y^2$, $xy^2$, $y^8)$. 
Set \(f_0 = x^6\), \(f_i = x^{t-i}y^2\) for 
\(i=1,\dots,4\), \(f_5 = xy^2-y^5\) and \(f_6=y^8\). Note that $\{f_0,\dots,f_6\}$ is a $\otau$-enhanced standard basis of $J$ with $\Ltlex(f_i)=\Lt(f_i)=x^{t-i}y^{m_i}$.
The matrix \(H+N\) associated to 
\(\{f_0,\dots,f_6\}\) is the following:
\[
\left(\begin{array}{cccccc}
y^2 & 0 & 0 & 0 & 0 & 0\\
-x & 1 & 0 & 0 & 0 & 0\\
0 & -x & 1 & 0 & 0 & 0\\
0 & 0 & -x & 1 & 0 & 0\\
0 & 0 & 0 & -x & 1 & 0\\
0 & 0 & 0 & 0 & -x-y^3 & y^6\\
0 & 0 & 0 & 0 & -1 & -x+y^3\\
\end{array}\right).
\]
The matrix \(N \in \mathcal{N}(E)\) is strictly lower triangular, 
but since \(\deg(n_{6,5})=3 \ge d_5=0\) and \(\deg(n_{7,5})=0\ge d_5=0\),
we see that \(N \notin T_0(E)\). By performing the reduction moves \((6,5)\) and \((7,5)\), we obtain the canonical Hilbert-Burch matrix $H+N_0$ of $J$, with $N_0\in\mathcal{M}(E)$:
\[M_0=H+N_0=\left(\begin{array}{cccccc}
y^2 & 0 & 0 & 0 & 0 & 0\\
-x & 1 & 0 & 0 & 0 & 0\\
0 & -x & 1 & 0 & 0 & 0\\
0 & 0 & -x & 1 & 0 & 0\\
0 & 0 & 0 & -x & 1 & 0\\
0 & 0 & 0 & 0 & -x & y^6\\
0 & 0 & 0 & 0 & 0 & -x+y^3\\
\end{array}\right).\]

\end{example}

There is a class of monomial ideals \(E\) such that any ideal with leading term ideal \(E\) is under the hypothesis of \Cref{CVmatrix}:

\begin{lemma}\label{lexGrobner} 
Let \(E = (x^t,x^{t-1}y^{m_1},\dots, y^{m_t})\) be a monomial ideal such that 
\begin{equation}\label{lexGBcond}
m_j-j-1 \le m_i-i    \mbox{ for all } j<i.
\end{equation} 
Then the reduced $\otau$-enhanced standard basis of $J \in V(E)$ is a Gr\"obner basis of $I=J\cap P$ with respect to the lexicographical term ordering and $\Ltlex(I)=E$. 
\end{lemma}

\begin{proof} 
Let $\{f_i\}_{i \in \I}$ with \(\I \subset \{0,\dots,t\}\) be the unique reduced $\otau$-enhanced standard basis of $J$ with \(\Lt(f_i)=x^{t-i}y^{m_i}\). There are two steps in this proof:

\medskip

\noindent
$(i)$ $\Ltlex(f_i)=x^{t-i}y^{m_i}$ for any $i \in \I$. 

Let us suppose that $\Ltlex(f_i)=x^ky^l\neq x^{t-i}y^{m_i}$. Since $x^{t-i}y^{m_i}\in\supp(f_i)$, then 
\[x^ky^l>_{\lex}x^{t-i}y^{m_i}\]
and hence there are two possible situations:

\noindent
\emph{Case I:} $k=t-i$ and $l>m_i$. $\Ltlex(f_i)=x^{t-i}y^l$ is in the support of $\tail_{\otau}(f_i)$ but $x^{t-i}y^l\in E$, which contradicts the reducedness hypothesis on $\{f_j\}_{j \in \I}$. 

\noindent
\emph{Case II:} $k>t-i$. Then we can set $k=t-j$ for some $0<j<i$. Since $\Ltlex(f_i)=x^{t-j}y^l$ and $\Lt(f_i)=x^{t-i}y^{m_i}$, then 
$$t-i+m_i=\deg(x^{t-i}y^{m_i})\leq\deg(x^{t-j}y^l)=t-j+l.$$
If there is an equality on the degree, the local term ordering is equal to the lexicographical term ordering, hence $\Lt(f_i)=x^{t-j}y^l$ and we reach a contradiction. Therefore, we have $t-i+m_i<t-j+l$.
If $l\geq m_j$, the argument of Case I holds. 
Thus, we obtain the following sequence of strict inequalities
\[
t-i+m_i < t-j+l < t-j+m_j.
\]
It is equivalent to \[m_i-i+1 \le l-j \le m_j-j-1\]
But by assumption 
\( m_j-j-1 \le m_i-i\), which leads to a contradiction.

\medskip

$(ii)$ $\{f_i\}_{i \in \I}$ is a Gr\"obner basis of $I$ with respect to the lexicographical term ordering.

Since $\{f_i\}_{i \in \I}$ is a subset of $I$, $E=(\Ltlex(f_i))_{i \in \I} \subset \Ltlex(I)$.
We can check that $\Ltlex(I)=E$ by looking at the dimensions.
From $R/J\cong P/I$, it follows that 
\[\dim_\res(P/\Ltlex(I))=\dim_\res(P/I)=\dim_\res(R/J)=\dim_\res(P/\Lt(J))=\dim_\res(P/E)\]
and hence the inclusion $E\subset \Ltlex(I)$ becomes an equality.
\end{proof}

\begin{remark}
Since for lex-segment ideals the sequence \((m_i-i)_i\) 
is strictly increasing, lex-segment ideals satisfy \ref{lexGrobner}. 
But the class of ideals is bigger. For example ideals with equality \(m_i = m_{i+1}\) for exactly one \(i\) satisfy this condition too.
\end{remark}

\begin{theorem}\label{ParamLex}
Let $E=(x^t,\dots,x^{t-i}y^{m_i},\dots,y^{m_t})$ be a lex-segment ideal (or an ideal satisfying condition (\ref{lexGBcond})). Let $H$ be the canonical Hilbert-Burch matrix of $E$. Then the restriction of the map \(\varphi\) from \Cref{Param} to \(\mathcal{M}(E)\)

$$\begin{array}{rrcl}
\varphi: & \mathcal{M}(E) & \longrightarrow & V(E)\\
& N & \longmapsto & I_t(H+N)
\end{array}$$

\noindent
is a bijection.
\end{theorem}

\begin{proof}
The map \(\varphi\) is well-defined by \Cref{well}. \Cref{lexGrobner} and \Cref{CVmatrix} ensure the existence of a unique matrix $N\in\mathcal{M}(E)$ such that $J=I_t(H+N)$. 
\end{proof}

Note that when $E$ is an ideal satisfying (\ref{lexGBcond}), then the set $\mathcal{M}(E)$ has a simple description. It is formed by matrices of size $(t+1)\times t$ with entries in $\res[y]$ such that 

$$n_{i,j}=
\left\lbrace\begin{array}{ll}
0, & i\leq j;\\
c_{i,j}^{v_{i,j}}y^{v_{i,j}}+c_{i,j}^{v_{i,j}+1}y^{v_{i,j}+1}+\dots+c_{i,j}^{d_j-1}y^{d_j-1}, & i > j;
\end{array}\right.
$$

where \(v_{i,j}:=\max(u_{i,j},0)\).

\begin{corollary}\label{cor:ParamLex} 
Let $E$ be the lex-segment ideal $(x^t,x^{t-1}y^{m_1},\dots,y^{m_t})$ 
(or an ideal satisfying condition (\ref{lexGBcond})) 
with degree matrix $U=(u_{i,j})$, 
\(v_{i,j}=\max(u_{i,j},0)\) and $d_j=m_j-m_{j-1}$ for any $1\leq i\leq t+1$ and $1\leq j\leq t$. 
Then $V(E)$ is an affine space of dimension $\mathbf{N}$, where 
$$\mathbf{N}=\sum_{2\leq j+1\leq i\leq t+1}\left(d_j-v_{i,j}\right).$$
\end{corollary}

Let us show the details of the parametrization of the Gr\"obner cell $V(E)$ as an affine space $\mathbb{A}_\res^\mathbf{N}$ with an example:

\begin{example}\label{ex:Paramlex}{\emph{Gr\"obner cell of a lex-segment ideal.}} Consider the lex-segment ideal $L=(x^3,x^2y,xy^3,y^5)$. 
By \Cref{ParamLex}, any canonical Hilbert-Burch matrix $M=H+N$, with $N \in \mathcal{M}(L)$, associated to an ideal $J \in V(L)$ is of the form
$$M=\left(\begin{array}{ccc}
y & 0 & 0\\
-x & y^2 & 0\\
c_{3,1}^0 & -x+c_{3,2}^1y & y^2\\
c_{4,1}^0  & c_{4,2}^0+c_{4,2}^1y  & -x+c_{4,3}^1y
\end{array}\right).$$

We identify any ideal $J=I_3(M)$ with the point 
$$p_J=(c_{3,1}^0,c_{4,1}^0,c_{3,2}^1,c_{4,2}^0,c_{4,2}^1,c_{4,3}^1)\in\mathbb{A}_\res^6.$$

In other words, $V(L)$ can be identified with the affine space $\mathbb{A}_\res^6$. Note that the point at the origin in $\mathbb{A}^6_\res$ corresponds to the monomial ideal $L$.
\end{example}

\medskip
\begin{corollary} \label{Cor:hilbertfunction} 
Assume $\ch(\res)=0$ and let $h$ be an admissible Hilbert function. 
Let $L=\Lex(h)$ be the unique lex-segment ideal such that $\HF_{R/L}=h$.
Then any ideal $J\subset R$ such that $\HF_{R/J}=h$ is of the form $I_t(H+N)$, for some $N\in \mathcal{M}(L)$, after a generic change of coordinates.
\end{corollary}

\begin{proof}
It follows from \Cref{ParamLex} and the fact that for any $J\subset R$ such that $\HF_{R/J}=h$ it holds $\Lex(h)=\Gin_\otau(J)$. Here $\Gin_\otau(J)$ is the extension to the local case defined in \cite[Theorem–Definition~1.14]{Ber09} of the usual notion of generic initial ideal.
\end{proof}

\begin{example}\label{ex:CV-vs-local}\emph{Two stratifications of $\Hilb^3(\res[\![x,y]\!])$.}
There are three monomial ideals of colength 3 in two variables: $E_1=(x,y^3)$, $E_2=(x^2,xy,y^2)$ and $E_3=(x^3,y)$. The punctual Hilbert scheme $\Hilb^3(\res[\![x,y]\!])$ can be stratified into three corresponding Gr\"obner cells that depend on the term ordering that we choose. The following table describes the ideals that we find in each Gr\"obner cell with respect to the lexicographical term odering, namely $V_2(E_i)$, and the induced local term ordering, namely $V(E_i)$, with $i=1,2,3$. Recall that $V_2(E_i)$ is the affine space in Conca-Valla parametrization introduced in \Cref{V2} that only considers $\m$-primary ideals in the polynomial ring, hence it provides a proper stratification of $\Hilb^3(\res[\![x,y]\!])$. 

\medskip

\begin{center}
\begin{tabular}{c||c|c|c}
 $E_i$ & $E_1=(x,y^3)$  & $E_2=(x^2,xy,y^2)$ & $E_3=(x^3,y)$\\
\hline
\hline
$\HF_{R/E_i}$ & $(1,1,1)$ & $(1,2)$ & $(1,1,1)$\\
\hline
\hline
$\otau$ & $J=(x,y^3+c_2y^2+c_1y)$ & $J=(x^2,xy,y^2)$ & $J=(x^3,y+cx^2)$\\
\hline
 $\tau=\lex$  & $I=(x,y^3+c_2y^2+c_1y)$   &  $I=(x^2+cy,xy,y^2)$ & $I=(x^3,y)$
\end{tabular}
\end{center}

\medskip

The second row of the table displays the Hilbert function of the local ring $(R/E_i,\n)$ as a sequence of natural numbers such that the element in position $t$ (starting at position 0) is the dimension of the $\res$-vector space $\n^t/\n^{t+1}$ and zero-dimensional vector spaces are omitted.

Consider an ideal $I\in V_2(E_2)$ with $c\neq 0$ and note that $y$ belongs to the initial ideal $I^{\ast}$. 
Therefore $E_2=\Ltau(I)\neq \Ltau(I^{\ast})$, hence the Hilbert functions of the local rings $P/I$ and $P/E_2$ differ.
In other words, the Gr\"obner cell $V_2(E_2)$ contains ideals with different Hilbert functions.

On the other hand, the local Gr\"obner cell $V(E_2)$ consists of a single point. 
By construction, such cells will always preserve the Hilbert function. In this sense we say that the parametrization given in \Cref{ParamLex} is compatible with the local structure. 
\end{example}

In the general case, we have a surjective map $ \varphi: \mathcal{N}(E)_{\le s} \to V(E)$.
Restricting to \(\mathcal{M}(E)\) we get an injection to \(V(E)\), but if \(E\) does not satisfy condition (\ref{lexGBcond}) the map \(\varphi\) is not surjective anymore.

\begin{lemma} \label{nonsurj}
If \(E\) does not satisfy condition (\ref{lexGBcond}), then there exists
\(J \in V(E)\) such that \(\Ltlex(J \cap P) \not= E\).
\end{lemma}

\begin{proof}
Since condition (\ref{lexGBcond}) is not satisfied, there exist \(j<i\) 
such that 
\begin{equation}\label{eq:not}
m_j-j-1 > m_i-i.     
\end{equation}
Take \(i'= \max\{l \ | \ m_i=m_l\}\)
and \(j' = \min\{l \ | \ m_j = m_l\}\), then  
\[m_{j'}-j' -1 > m_j-j-1 > m_i-i > m_{i'}-i'.\]
Replace \(i\) with \(i'\) and \(j\) with \(j'\).
Note that (\ref{eq:not}) still holds and now additionally \(d_j\ge 1\) and \(d_{i+1} \ge 1\).

Set \(f_k = x^{t-k}y^{m_k}\) for \(k \in \{0,\dots,t\}\setminus \{i\}\) and
\(f_i = x^{t-i}y^{m_i} + x^{t-j}y^{m_j-1}\). Consider the ideal
\(J = (f_0,\dots, f_t)\) of $R$.
Clearly, \(\Ltlex(f_i) = x^{t-j}y^{m_j-1} \notin E\), thus 
\(\Ltlex(J \cap P) \not= E\). 

Now we need to prove that $\Lt(J)=E$.
From (\ref{eq:not}) we have \(t-i+m_i<t-j+m_j-1\), so
\(\Lt(f_i) = x^{t-i}y^{m_i}\). 
The polynomial \(f_i\) cannot be reduced by the other (monomial) generators.

The \(S\)-polynomials are 
\[
S_l = \begin{cases}
-x^{t-j+1}y^{m_j-1}, & l = i; \\
x^{t-j} y^{m_j-1+d_{i+1}},  & l = i+1; \\
0, & \mbox{otherwise.}
\end{cases}
\]

If $i<t$, check that \(S_{i} = y^{d_j-1} f_{j-1}\) and \(S_{i+1} = y^{d_{i+1}-1} f_j \).
Then the matrix \(N\) has only two non-zero entries \(n_{j,i}=y^{d_j-1}\) and
\(n_{j+1,i+1}=-y^{d_{i+1}-1}\).
If \(i=t\) there is only one non-zero \(S\)-polynomial. In any case, one can check that \(N \in \mathcal{N}(E)\).
Thence, \(\{f_0,\dots,f_t\}\) forms a \(\otau\)-enhanced standard basis and
\(J \in V(E)\).
\end{proof}

\begin{example} {\emph{$\mathcal{M}(E)\rightarrow V(E)$ not surjective.}}
Consider $E=(x^6,xy^2,y^8)$ as in \Cref{Ex:step1}.  
\(E\) does not satisfy condition (\ref{lexGBcond}) because for \((i,j)=(5,1)\) we have \(m_1-1-1 = 0 > m_5-5 = 2-5 = -3\). The ideal \(J\) from \Cref{nonsurj} in this case is generated by the monomials \(x^{6-k}y^{m_k}\) for \(k=0,\dots,4,6\)
and \(xy^2+x^5y\). 
\(J \cap P\notin V_0(E)\) because $\Ltlex(J \cap P)=(x^6,x^5y,x^2y^2,xy^3,y^8)$. 
Therefore, \(J \notin \varphi(\mathcal{M}(E))\).
\end{example}

Several computations
suggest us how to replace $\mathcal{N}(E)$ in \Cref{Param} in order to obtain a bijection.

We define the subset \((\res[y]_{<\underline{d}})^{(t+1) \times t} \subset \res[y]^{(t+1) \times t}\)
as matrices where the entries satisfy the following degree conditions: \[
\deg(n_{i,j}) < \begin{cases} 
d_i, & i\le j;\\
d_j, & i > j. 
\end{cases}
\]

\begin{conjecture}\label{Conj}
Let $E$ be a monomial ideal. Then the set $\mathcal{N}(E)_{< \underline{d}}:= \mathcal{N}(E) \cap (\res[y]_{<\underline{d}})^{(t+1) \times t}$
parametrizes $V(E)$.
\end{conjecture}

For any ideal $E$ satisfying condition (\ref{lexGBcond}) the sets $\mathcal{N}(E)_{< \underline{d}}$ and $\mathcal{M}(E)$ coincide. By \Cref{ParamLex}, the conjecture is true for such $E$, which includes lex-segment ideals. 
For general \(E\), we have an inclusion
\(\mathcal{M}(E) \subset \mathcal{N}(E)_{< \underline{d}}\).
Moreover, the matrix \(N\) constructed in the proof of 
\Cref{nonsurj}, which is not in \(\mathcal{M}(E)\), can also be transformed to a matrix in 
\(\mathcal{N}(E)_{< \underline{d}}\) via reduction moves.

\begin{example}{\emph{$E$ not satisfying (\ref{lexGBcond}) where \Cref{Conj} holds.}} Consider the monomial ideal $E=(x^4,y^2)$. It can be proved that any $J\in V(E)$ is of the form $J=(x^4+ax^3y,y^2+bx^3+cx^3y+dx^2y)$. The $S$-polynomials of the standard basis
$$\begin{array}{l}
f_0=x^4+ax^3y\\
f_1=x^3y^2\\
f_2=x^2y^2\\
f_3=xy^2+(d-ab)x^3y+(ad-a^2b)x^2y^2\\
f_4=y^2+bx^3+cx^3y+dx^2y+(ad-a^2b)xy^2+(a^3b^2-2a^2bd+ad^2+c)x^3y
\end{array}$$

\noindent
of $J$ give the matrix $M=H+N$, with $N\in\mathcal{N}(E)_{< \underline{d}}$ and $I_4(M)=J$, satisfying the conjecture:
$$M=\left(\begin{array}{cccc}
y^2  &  0 & (d-ab)y & b+(a^3b^2-2a^2bd+ad^2+c)y\\
-x-ay & 1 & 0 & 0\\
0 & -x & 1 & 0\\
0 & 0 & -x & 1\\
0 & 0 & 0 & -x
\end{array}\right).$$
\end{example}

\section{Applications to the construction of Gorenstein rings}\label{S:Gorenstein}

Let us assume that $\res$ is a field of characteristic 0. The explicit description of the affine variety $V(L)$ given by \Cref{ParamLex} allows us to parametrize Gorenstein rings $R/J$ with a given Hilbert function $h$ up to a generic change of coordinates. It is enough to consider those Gorenstein ideals $J$ that arise as a deformation of the unique lex-segment ideal $L=\Lex(h)$ associated to $h$. We will now see that the subset $V_G(L)$ of all Gorenstein ideals in $V(L)$ has the structure of a quasi-affine variety. Note that in codimension 2, the Gorenstein condition is equivalent to being a complete intersection ideal. 

\begin{proposition}\label{rkCriteria} Let $L$ be a lex-segment ideal and let $J$ be an ideal with $\Lt(J)=L$. Let $H$ and $M=H+N$ be the canonical Hilbert-Burch matrices of $L$ and $J$, respectively. Then $J$ is Gorenstein if and only if 
the entries \(n_{i+2,i}\) for \(i=1,\dots, t-1\) of $N$ 
consist of polynomials in $y$ with non-zero constant terms. 
\end{proposition}

\begin{proof} In codimension 2, $J$ is Gorenstein if and only if it is minimally generated by 2 elements. Let $\overline{M}$ be the matrix whose entries are the classes of the entries of $M$ in $R/\m$. By \cite[Lemma 2.1]{Ber09}, $J$ is Gorenstein if and only if  $\rk(\overline{M})=t-1$. It can be checked easily that this is equivalent to $c_{3,1}^0 c_{4,2}^0\cdots c_{t+1,t-1}^0\neq 0$, where $c_{i,i-2}^0$ is the constant term of the entry $n_{i,i-2}$ of $N$.
\end{proof}

\begin{remark} \Cref{rkCriteria} provides a method of determining whether a lex-segment ideal $L$ admits Gorenstein deformations by looking at the degree matrix $U$ of the canonical Hilbert-Burch matrix $H$ of $L$. Gorenstein ideals are admissible if and only if $u_{i,i-2}\leq 0$ for any $3\leq i\leq t+1$. See \cite{Ber09} for details on what the admissible Hilbert functions for Gorenstein rings of codimension 2 are. 
\end{remark}

\begin{example}\label{ex:VGL}\emph{Parametrization of Gorenstein deformations of a lex-segment ideal.} Consider $L=(x^3,x^2y,xy^3,y^5)$. From \Cref{ex:Paramlex} we have
$$\overline{M}=\left(\begin{array}{ccc}
0 & 0 & 0\\
0 & 0 & 0\\
c_{3,1}^0 & 0 & 0\\
c_{4,1}^0 & c_{4,2}^0 & 0
\end{array}\right).$$

\noindent
By \Cref{rkCriteria}, $J=I_3(M)$ is Gorenstein if and only if $c_{3,1}^0c_{4,2}^0\neq 0$. Then the set of Gorenstein ideals $J$ with $\Lt(J)=L$ can be identified with $\mathbb{A}^6_\res\backslash\mathbb{V}(c_{3,1}^0c_{4,2}^0)$.
\end{example}

\begin{corollary}\label{cor:VGL} Let $L$ be a lex-segment ideal. The set $V_G(L)$ of Gorenstein ideals $J$ such that $\Lt(J)=L$ is a quasi-affine variety.
\end{corollary}

\begin{remark}
\Cref{cor:VGL} is a generalization of the procedure given in \cite[Remark 4.7]{RS10} by Rossi and Sharifan to explicitly construct a Gorenstein ring $J$ whose resolution is obtained by consecutive and zero cancellation of the resolution of $L=\Lex(h)$. 
\end{remark}

\bigskip

The parametrization of Gorenstein ideals can also be used to find Gorenstein Artin rings $G=R/J$ that are as close as possible to a given Artin ring $A=R/I$. See \cite{Ana08},\cite{EH18},\cite{EHM20} for more details on this problem.

\begin{definition} We call the Artin Gorenstein ring $G=R/J$ a \textbf{minimal Gorenstein cover} of the Artin ring $A=R/I$ if $J\subset I$ and $\dim_\res G-\dim_\res A$ is minimal among all Artin Gorenstein rings mapping onto $A$. The difference $\dim_\res G-\dim_\res A$ is called the \textbf{Gorenstein colength} of $A$, denoted by $\gcl(A)$.
\end{definition}

Let us show through an example how we can find such Gorenstein covers using the canonical Hilbert-Burch matrices provided by \Cref{ParamLex}: 

\begin{example}\emph{Parametrization of minimal Gorenstein covers of $A=R/I$ arising from a lex-segment ideal.} Consider the ideal $I=(x^3-2xy^2,x^2y-2y^3,y^3)$ with Hilbert function $(1,2,3,1)$. The sequence $h=(1,2,3,2,1)$ corresponds to the Hilbert function of smallest length that admits Gorenstein ideals $J$ where the inclusion $J\subset I$ is possible a priori. The lex-segment ideal associated to $h$ is our running example $L=(x^3,x^2y,xy^3,y^5)$, see \Cref{ex:Paramlex} and \Cref{ex:VGL}. 

On one hand, the inclusion condition $J\subset I$ can be described by a normal form computation of the generators of $J\in V(L)\simeq\mathbb{A}_\res^6$ with respect to a standard basis of $I$. The point $p_J=(c_{3,1}^0,c_{4,1}^0,c_{3,2}^1,c_{4,2}^0,c_{4,2}^1,c_{4,3}^1)\in\mathbb{A}_\res^6$ satisfies the inclusion property if and only if it belongs to the affine variety $\mathbb{V}(-c_{3,1}^0+c_{3,2}^1c_{4,3}^1-c_{4,2}^0+2,c_{3,1}^1+c_{4,3}^1)\subseteq\mathbb{A}_\res^6$. 

On the other hand, $V_G(L)\simeq\mathbb{A}_\res^6\backslash\mathbb{V}(c_{3,1}^0c_{4,2}^0)$. Therefore, $J$ is a Gorenstein cover of $A$ if and only if $p_J\in \mathbb{V}(-c_{3,1}^0+c_{3,2}^1c_{4,3}^1-c_{4,2}^0+2,c_{3,1}^1+c_{4,3}^1)\backslash\mathbb{V}(c_{3,1}^0c_{4,2}^0)$.

For instance, the point $(1,0,0,1,0,0)\in\mathbb{A}_\res^6$ corresponds to the Gorenstein cover $G=R/(x^2y-y^3,x^3-2xy^2)$ of $A$. In particular, we proved that $\gcl(A)=2$.
\end{example}

\begin{corollary}\label{cor:wow} The set of Gorenstein covers of $G=R/J$ of $A=R/I$ that arise from a deformation of a lex-segment ideal $L$, namely $\Lt(J)=L$, is a quasi-affine variety.
\end{corollary}

\begin{remark}
Not all minimal Gorenstein covers $G=R/J$ of $A=R/I$ come from deformations of a lex-segment ideal. Consider the ideal $I=(x^3,xy^2,y^3)$. It can be checked that $J=(x^3,y^3)$ corresponds to a minimal Gorenstein cover of $A$ with Hilbert function $h=(1,2,3,2,1)$. The lex-segment ideal associated to $h$ is $\Lex(h)=(x^3,x^2y,xy^3,y^5)$, however no ideal in $V(\Lex(h))$ will provide a Gorenstein cover of $A$. Indeed, $\Lex(h)$ is not contained in $I$, hence none of its deformations will be. 

Therefore, to find minimal Gorenstein covers we need to look into all cells $V(E)$ such that $\HF_{R/E}=h$.
The reason behind this is that the inclusion condition $J\subset I$ is not preserved after a generic change of coordinates on $J$.  

The surjectivity of \Cref{polysurj} is enough to detect the existence of minimal Gorenstein covers coming from a non-lex-segment ideal $E$ but in order to compute the quasi-projective variety of all minimal Gorenstein covers (see \cite[Theorem 4.2]{EHM20}) we need a parametrization of $V(E)$. 
Examples can be found in \cite{Hom19}.
\end{remark}

\section*{Acknowledgements}
We want to thank Alexandru Constantinescu for suggesting the problem 
and for helpful discussions, as well as for recommending to the second author a stay at the Universit\`a Degli Studi di Genova. The first author wants to thank Joan Elias for encouraging her to do a research stay with Maria Evelina Rossi. We also want to thank Bernd Sturmfels for his advice.

We want to give special thanks to Maria Evelina Rossi for hosting us in Genova, answering many questions, giving useful hints and commenting on several versions of this manuscript, and especially for suggesting us to work on this problem together.

We finally would like to thank the reviewer for many insightful comments that improved the paper.

The first author was partially supported by MTM2016-78881-P, BES-2014-069364 and EEBB-I-18-12915.
Travel of the second author to Genova was supported by MIUR-DAAD Joint
Mobility Program 57267452.
    
\bibliographystyle{plain}
\bibliography{references}
\end{document}